\documentclass{article}

\usepackage[utf8]{inputenc}
\usepackage{mathtools,amsfonts,amssymb,amsthm,tikz-cd,tikz}
\newtheorem{theorem}{Theorem}
\newtheorem{lemma}[theorem]{Lemma}
\newtheorem{corollary}{Corollary}[theorem]
\newtheorem*{remark}{Remark}
\newtheorem{proposition}{Proposition}
\theoremstyle{definition}
\newtheorem{definition}{Definition}[section]
\title{An upper bound and criteria for the Galois group of weighted walks with rational coefficients in the quarter plane}
\date{August 2020}
\author{Ruichao Jiang, Javad Tavakoli, Yiqiang Zhao}
\usepackage{graphicx}

\begin{document}
\maketitle
\begin{abstract}
    Using Mazur's theorem on torsions of elliptic curves, an upper bound 24 for the order of the finite Galois group $\mathcal{H}$ associated with weighted walks in the quarter plane $\mathbb{Z}^2_+$ is obtained. The explicit criterion for $\mathcal{H}$ to have order 4 or 6 is rederived by simple geometric argument. Using division polynomials, a recursive criterion for $\mathcal{H}$ having order $4m$ or $4m+2$ is also obtained. As a corollary, explicit criterion for $\mathcal{H}$ to have order 8 is given and is much simpler than the existing method.
\end{abstract}
\section{Introduction}
Counting lattice walks is a classic problem in combinatorics. A combinatoric walk with nearest-neighbour step length can be seen as a weighted walk with weight 1 for the allowed directions and weight 0 for the forbidden directions. If a multiple step length requirement is allowed, a combinatoric walk can be seen as a weighted walk with integer weights. Without loss of generality, for a weighted walk, we may assume that the weights sum to 1 by normalization. If we allow the weights of a walk to take arbitrary non-negative real values that sum to 1, then we arrive at the realm of probabilistic walks in the quarter plane. So a weighted walk is the same thing as a probabilistic walk and a weighted walk with rational weights is the same thing as a combinatoric walks with different step lengths in different directions.

In the probabilistic scenario, an approach called the "kernel method" has been well developed and summarized in the book \cite{yellowbook}. In the kernel method, Malyshev \cite{Malyshev} defined a group $\mathcal{H}$, called the Galois group associated with any walk in $\mathbb{Z}^2_+$. The finiteness of $\mathcal{H}$ turns out to be important. Here are some applications of $\mathcal{H}$
\begin{enumerate}
    \item For the 2-demands queueing system, Flato and Hahn \cite{flato1,flato2} exploited the finiteness of $\mathcal{H}$ to obtain an exact formula for the stationary distribution.
    \item The generating function of the walk satisfies some differential equation if and only if $\mathcal{H}$ has finite order. Moreover the generating function is algebraic if and only if the orbit sum is zero. See Theorem 42 in \cite{dreyfus}.
\end{enumerate}
Bousquet-M\'elou \cite{Bousquet} showed that for combinatoric walks with nearest-neighbour step length in the quarter plane, $\mathcal{H}$ can have order 4, 6, or 8, if $\mathcal{H}$ has finite order. For a weighted walk, Kauers and Yatchak found three walks with order $10$ \cite{kauers}.

In this paper, we only consider the generic case when the kernel of the walk determines genus 1 surfaces. We give 24 as an upper bound on the finite order of $\mathcal{H}$ when the weights of the walk are rationals. In particular, this result says that if the order of $\mathcal{H}$ is finite, then it cannot be arbitrarily large. The following list summarizes different objects considered in the paper and also serves as an outline of the proof.
\begin{enumerate}
    \item A biquadratic polynomial $Q(x,y)$ defines a connected real curve $Q\subset{\mathbb{R}^2}$. The composition of the horizontal and the vertical switches is called a QRT map $\delta$ on $Q$.
    \item By going to complex numbers, $Q(x,y)$ defines a Riemann surface, also called $Q\subset\mathbb{C}^2$. The Abel-Jacobi map $\mathcal{J}$ determines a lattice $\Lambda$ generated by $\omega_1,\omega_2\in\mathbb{C}$, unique up to the modular group $\text{PSL}(2,\mathbb{Z})$ action, such that $Q\cong\mathbb{C}/\Lambda$.
    \item The Weierstrass function $\wp$ and its derivative $\wp'$ can be used to construct a map $\mathcal{J'}^{-1}$, an "inverse" of the Abel-Jacobi map $\mathcal{J}$. It is not an actual inverse because the image of $\mathcal{J'}^{-1}$ is not $Q$ but an elliptic curve $E$ in the Weierstrass normal form.
    \item Both $\mathcal{J}$ and $\mathcal{J'}^{-1}$ are defined analytically. However, the composition of them turns out to be a polynomial map. So if we start with a $Q(x,y)$ with rational coefficients, we obtain an elliptic curve $E$ with rational coefficients.
    \item Moreover, the QRT map $\delta$ induces an addition by a rational point on $E$.
    \item So the Mazur's theorem applies and the bound is obtained.
\end{enumerate}
The organization of the paper is as follows: Section 2 covers 1-3 in the above list, Section 3 covers 4-5, Section 4 covers 6. Section 5 covers criteria for $\mathcal{H}$ to have order $4m$ or $4m+2$. Section 6 is discussion.

\section{Preliminary}\label{pre}
In this section, we provide the preliminaries that are needed for our main result.
\subsection{The model}
We shall consider walks in $\mathbb{Z}^2_{+}$ with step length limited to 1 (nearest-neighbor) and the walk is considered to be homogeneous, that is, the transition probabilities $p_{i,j} (-1\leq i,j\leq1)$'s are independent of the current place.
\begin{figure}[ht]
  \centering
  \resizebox{0.5\textwidth}{!}{%
    \begin{tikzpicture}
        \coordinate (Origin)   at (2,2);
        \coordinate (XAxisMin) at (-1,0);
        \coordinate (XAxisMax) at (6,0);
        \coordinate (YAxisMin) at (0,-1);
        \coordinate (YAxisMax) at (0,6);
        \draw [thin, gray,-latex] (XAxisMin) -- (XAxisMax);
        \draw [thin, gray,-latex] (YAxisMin) -- (YAxisMax);
        \clip (-2,-2) rectangle (5cm,5cm);
        \coordinate (p11) at (3,3);
        \coordinate (p10) at (3,2);
        \coordinate (p1-1) at (3,1);
        \coordinate (p0-1) at (2,1);
        \coordinate (p-1-1) at (1,1);
        \coordinate (p-10) at (1,2);
        \coordinate (p-11) at (1,3);
        \coordinate (p01) at (2,3);
        \foreach \x in {0,1,...,7}{
            \foreach \y in {0,1,...,7}{
                \node[draw,circle,inner sep=2pt,fill] at (2*\x,2*\y) {};
            }
        }
        \draw [ultra thick,-latex] (Origin)
            -- (p11) node [above right] {$p_{1,1}$};
        \draw [ultra thick,-latex] (Origin)
            -- (p10) node [right] {$p_{1,0}$};
        \draw [ultra thick,-latex] (Origin)
            -- (p1-1) node [below right] {$p_{1,-1}$};
        \draw [ultra thick,-latex] (Origin)
            -- (p0-1) node [below] {$p_{0,-1}$};
        \draw [ultra thick,-latex] (Origin)
            -- (p-1-1) node [below left] {$p_{-1,-1}$};
        \draw [ultra thick,-latex] (Origin)
            -- (p-10) node [left] {$p_{-1,0}$};
        \draw [ultra thick,-latex] (Origin)
            -- (p-11) node [above left] {$p_{-1,1}$};
        \draw [ultra thick,-latex] (Origin)
            -- (p01) node [above] {$p_{0,1}$};
  \end{tikzpicture}
  }
  \caption{The model. $p_{0,0}$ is not shown.}
\end{figure}
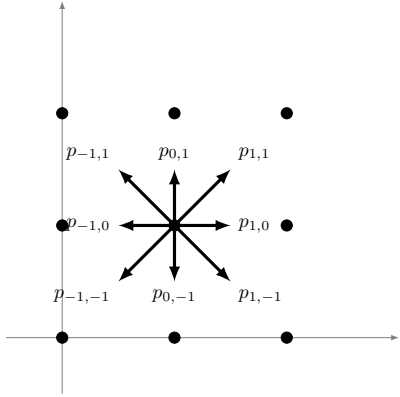
To determine the stationary distribution $\{\pi_{ij}, i,j\in\mathbb{N}\}$ of the walk, following \cite{yellowbook}, the generating function method is applied. The generating function
\begin{equation}
    \pi(x,y)=\sum_{i,j\geq1}\pi_{ij}x^{i-1}y^{j-1}
\end{equation}
satisfies the following functional equation:
\begin{equation}\label{functional}
    Q(x,y)\pi(x,y)=q(x,y)\pi(x)+\Tilde{q}(x,y)\Tilde{\pi}(y)+\pi_0(x,y),
\end{equation}
where
\begin{equation}
    Q(x,y)=xy\left(\sum_{i,j}p_{i,j}x^iy^j-1\right).
\end{equation}
Other terms reflects the boundary conditions on the random walk and do not enter our study.

$Q(x,y)$ is called the kernel of the random walk and is a biquadratic polynomial, i.e. both quadratic in $x$ and quadratic in $y$:
\begin{align*}\label{kernel}
\begin{split}
    Q(x,y)=&(p_{1,1}x^2+p_{0,1}x+p_{-1,1})y^2+(p_{1,0}x^2+(p_{0,0}-1)x+p_{-1,0})y\\
    &+p_{1,-1}x^2+p_{0,-1}x+p_{-1,-1}
\end{split}\\
    \coloneqq&a(x)y^2+b(x)y+c(x)\\
\begin{split}
    =&(p_{1,1}y^2+p_{1,0}y+p_{1,-1})x^2+(p_{0,1}y^2+(p_{0,0}-1)y+p_{0,-1})x\\
    &+p_{-1,1}y^2+p_{-1,0}y+p_{-1,-1}
\end{split}\\
    \coloneqq&\Tilde{a}(y)x^2+\Tilde{b}(y)x+\Tilde{c}(y)
\end{align*}
The partial discriminants of the kernel $Q$ is defined:
\begin{definition}[Partial discriminant]\label{partial-discriminant}
The partial discriminants of
$$Q(x,y)=a(x)y^2+b(x)y+c(x)=\Tilde{a}(y)x^2+\Tilde{b}(y)x+\Tilde{c}(y)$$
are defined as
\begin{equation}
    \Delta_1(y)\coloneqq\Tilde{b}^2(y)-4\Tilde{a}(y)\Tilde{c}(y),
\end{equation}
and
\begin{equation}
    \Delta_2(x)\coloneqq b^2(x)-4a(x)c(x).
\end{equation}
\end{definition}
By using complex variable and compactification, the kernel $Q$ determines a compact Riemann surface $Q$. Since $Q(x,y),\ x,y\in{\mathbb{C}}$ is biquadratic, $Q$ double covers the Riemann sphere $\hat{\mathbb{C}}$  with 4 branching point if the partial discriminant $\Delta_1(y)$ or equivalently $\Delta_2(x)$ has no multiple zeros \cite{yellowbook}. By Riemann-Hurwitz formula, the genus of $Q$ is
\begin{equation}
    g(Q)=2(g(\hat{\mathbb{C}})-1)+\frac{4}{2}(2-1)+1=1
\end{equation}
We shall assume that $Q$ has genus 1.
On $Q$, the following maps are defined:
\begin{definition}[Involutions and the QRT map]
The vertical switch $\xi$:
\begin{equation}\label{vertical}
    \xi(x,y)\coloneqq\left(x,-\frac{b(x)}{a(x)}-y\right).
\end{equation}
The horizontal switch $\eta$:
\begin{equation}\label{horizontal}
    \eta(x,y)\coloneqq\left(-\frac{\Tilde{b}(y)}{\Tilde{a}(y)}-x,y\right).
\end{equation}
The QRT map:
\begin{equation}
    \delta\coloneqq\eta\circ\xi.
\end{equation}
\end{definition}
The QRT map generates a group, called the Galois group associated with the random walk:
\begin{definition}[Galois group]
\begin{equation}
    \mathcal{H}\coloneqq\langle\xi,\eta\rangle
\end{equation}
\end{definition}
\begin{remark}
The reason why $\mathcal{H}$ is coined as Galois is essentially that Malyshev adopted a field-theoretic definition of the Riemann surface $Q$, where a point on $Q$ is defined as a discrete valuation on the function field $\mathbb{C}[x,y]/Q(x,y)$.
\end{remark}
As involutions, both $\xi$ and $\eta$ have order 2. However, the subgroup $\mathcal{H}_{0}\coloneqq\langle\delta\rangle$ can have finite or infinite order. Lemma 2.4.3 of \cite{yellowbook} says that $\mathcal{H}_0$ is a normal subgroup of $\mathcal{H}$ and $\mathcal{H}/\mathcal{H}_0$ is a group of order 2. Obviously $\mathcal{H}_0$ and $\mathcal{H}/\mathcal{H}_0$ have trivial intersection, hence
\begin{equation}\label{normal}
    \mathcal{H} = \mathcal{H}_0\rtimes\mathbb{Z}_2.
\end{equation}
The following invariants are useful later:
\begin{definition}[Eisenstein invariants]\label{eisen}
Let $f(x)=ax^4+4bx^3+6cx^2+4dx+e$ be a quartic polynomial. The Eisenstein invariants of $f$ are
\begin{equation}
    D(f)\coloneqq ae+3c^2-4bd,
\end{equation}
and
\begin{equation}
    E(f)\coloneqq ad^2+b^2e-ace-2bcd+c^3.
\end{equation}
\end{definition}
\subsection{Abel-Jacobi map}
Since the Riemann surface $Q$ has genus $1$, its topological structure is a torus. Hence $Q$ can support a nowhere vanishing vector field. Indeed, we have the following nowhere vanishing
\begin{definition}[Hamiltonian vector field]
The Hamiltonian vector field $v_{H}$ given by
\begin{equation}
v_H\coloneqq\frac{\partial Q}{\partial y}\frac{\partial}{\partial{x}}-\frac{\partial Q}{\partial x}\frac{\partial}{\partial{y}}
\end{equation}
\end{definition}
In fact, this can be used as a more direct proof that $Q$ has genus 1. Using Hamiltonian vector field, we can define a unique Abelian differential $\omega_{H}$, such that $(\omega_{H},v_H)=1$, where the parentheses denote the canonical pairing between vector fields and differential forms. Since $v_H$ is nowhere vanishing, so is $\omega_{H}$.

The complex structure of $Q$ is determined by a lattice $\Lambda$:
\begin{theorem}[Abel-Jacobi map]\label{abel-jacobi}
Let $Q$ be the Riemann surface determined by $Q(x,y)$ with genus $1$. Let $\omega$ be the Abelian differential determined by the Hamiltonian vector field $v_H$. Define a lattice $\Lambda=\langle\omega_1,\omega_2\rangle$ by complete elliptic integrals:
\begin{equation}
    \begin{cases}
      \omega_1=\int_{\gamma_1}\frac{\text{d}x}{\sqrt{\Delta_2(x)}}=\int_{\gamma_1}\frac{\text{d}y}{\sqrt{\Delta_1(y)}},\\
      \omega_2=\int_{\gamma_2}\frac{\text{d}x}{\sqrt{\Delta_2(x)}}=\int_{\gamma_2}\frac{\text{d}y}{\sqrt{\Delta_1(y)}},
    \end{cases}
\end{equation}
where $[\gamma_1]$ and $[\gamma_2]$ form a basis for $H_1(Q,\mathbb{Z})$. Let $p_0\in{Q}$ be an arbitrary point. Then the Abel-Jacobi map
\begin{equation}
    \mathcal{J}:Q\rightarrow{\mathbb{C}/\Lambda}
\end{equation}
given by incomplete elliptic integrals along a path from $p_0$ to $p$
\begin{equation}\label{imcomplete}
    p\mapsto\int_{p_0}^p\omega\ (\text{mod}\ \Lambda),
\end{equation}
is well-defined and does not depend on the path.
\end{theorem}
\begin{remark}
Under a modular group $\textrm{PSL}(2,\mathbb{Z})$ action, we may choose $\Lambda=\mathbb{Z}\oplus\tau\mathbb{Z}$, where $\tau=\pm\frac{\omega_2}{\omega_1}$. The $\pm$ sign here makes $\operatorname{Im}(\tau)>0.$
\end{remark}
\begin{remark}
$\mathcal{J}$ depends on the choice of $p_0$ but in a trivial way: A different choice of $p_0$ gives an integration constant and hence a translation on $\mathbb{C}/\Lambda$.
\end{remark}
The QRT map $\delta$ induces an addition on $\mathbb{C}/\Lambda$ via the Abel-Jacobi map:
\begin{proposition}\label{first-isomorphism}
The following diagram is commutative
\begin{center}
\begin{tikzcd}
Q\ar[r,"\delta"]\ar[d,"\mathcal{J}"]&Q\ar[d,"\mathcal{J}"]\\
\mathbb{C}/\Lambda\ar[r,"\delta^*"]&\mathbb{C}/\Lambda
\end{tikzcd}
\end{center}
where the map $\delta^*$ is given by
\begin{equation}
    \delta^*(z)=z+\omega_3\ (\textrm{mod}\ \Lambda)
\end{equation}
for $z\in\mathbb{C}/\Lambda$.
\end{proposition}
\subsection{Weierstrass normal form}
Since the field $K$, over which an elliptic curve $E$ is defined, plays a role, we will write $E(K)$ to mean that the polynomial defining $E$ has coefficients over $K$ and there exists a point on $E$ with coordinates in $K$.

Our goal is to transform a biquadratic curve $Q$ to an elliptic curve $E(\mathbb{C})$ in the Weierstrass normal form.
\begin{definition}[Weierstrass normal form]
An elliptic curve $E(\mathbb{C})$ is said to be in the Weierstrass normal form if $E(\mathbb{C})$ is defined by the polynomial
\begin{equation}
    y^2=4x^3-g_2x-g_3.
\end{equation}
\end{definition}
An elliptic curve carries a natural Abelian group law $+: E\times{E}\rightarrow{E}$. In the Weierstrass normal form, the Abelian group law can be described by the usual chord-tangent construction with the identity element being the point at infinity.

We have transformed the biquadratic curve $Q$ to the lattice $\Lambda$ via the Abel-Jacobi map $\mathcal{J}$. We need Weierstrass $\wp$ functions to transform $\Lambda$ to $E(\mathbb{C})$.
\begin{definition}[Weierstrass function]
The Weierstrass $\wp$ function for a lattice $\Lambda=\langle\omega_1,\omega_2\rangle$ is
\begin{equation}
    \wp(z)\coloneqq\frac{1}{z^2}+\sum_{\substack{(m,n)\neq(0,0)\\m,n\in\mathbb{Z}}}\left\{\frac{1}{(z-m\omega_1-n\omega_2)^2}-\frac{1}{(m\omega_1+n\omega_2)^2}\right\}
\end{equation}
with derivative being
\begin{equation}
    \wp'(z)=-2\sum_{m,n\in\mathbb{Z}}\frac{1}{(z-m\omega_1-n\omega_2)^3}.
\end{equation}
\end{definition}
Both series for $\wp$ and $\wp'$ converge locally uniformly in $\mathbb{C}-\Lambda$, hence they define holomorphic functions $\mathbb{C}-\Lambda$. $\wp$ and $\wp'$ are meromorphic in $\mathbb{C}$ and have pole of order 2 and 3 respectively on $\Lambda$.
\begin{definition}[Modular invariants]
The modular invariants for a lattice are defined by
\begin{equation}
    g_2(\Lambda)\coloneqq60\sum_{\substack{(m,n)\neq(0,0)\\m,n\in\mathbb{Z}}}\frac{1}{(m\omega_1+n\omega_2)^4}
\end{equation}
and
\begin{equation}
    g_3(\Lambda)\coloneqq140\sum_{\substack{(m,n)\neq(0,0)\\m,n\in\mathbb{Z}}}\frac{1}{(m\omega_1+n\omega_2)^6}.
\end{equation}
\end{definition}
The following theorem transform a lattice to an elliptic curve in Weierstrass normal form
\begin{theorem}[Inverse Abel-Jacobi map]\label{inverse}
The map
\begin{equation}
    \mathcal{J'}^{-1}:\mathbb{C}/\Lambda\rightarrow{E(\mathbb{C})}
\end{equation}
where $E(\mathbb{C})$ is an elliptic curve in the Weierstrass normal form given by
\begin{equation*}
    z\mapsto(\wp(z),\wp'(z))
\end{equation*}
is an isomorphism of analytic manifold and also a group isomorphism.
\end{theorem}
\begin{remark}
The map in the theorem is an inverse of the Abel-Jacobi map. First, we use $\mathcal{J}$ to transform $Q$ to $\mathbb{C}/\Lambda$. Then we use $\mathcal{J'}^{-1}$ to transform $\mathbb{C}/\Lambda$ to an elliptic curve in the Weierstrass normal form, which is isomorphic to $Q$ but in a different coordinate system. The composition $\mathcal{J'}^{-1}\circ\mathcal{J}$ has the effect of a change of coordinate.
\end{remark}
Now we have the following commutative diagram
\begin{center}
\begin{tikzcd}
Q\ar[r,"\delta"]\ar[d,"\mathcal{J}"]&Q\ar[d,"\mathcal{J}"]\\
\mathbb{C}/\Lambda\ar[r,"\delta^*"]\ar[d,"\mathcal{J}'^{-1}"]&\mathbb{C}/\Lambda\ar[d,"\mathcal{J}'^{-1}"]\\
E(\mathbb{C})\ar[r,"\delta^{**}"]&E(\mathbb{C})
\end{tikzcd}
\end{center}
Denote $\mathcal{J}'^{-1}(\omega_3)$ by $\Omega_3$. Since $\mathcal{J}'^{-1}$ is a group isomorphism, we have
\begin{equation}
    \delta^{**}(P)=P+\Omega_3
\end{equation}
for $P\in E(\mathbb{C})$.
\section{Complex to rational}
For an elliptic curve over $\mathbb{Q}$, following theorems hold.
\begin{theorem}[Mordell]\label{mordell}
The rational points on an elliptic curve form a finitely generated Abelian group.
\end{theorem}
\begin{theorem}[Mazur]\label{mazur}
For any $E(\mathbb{Q})$, the torsion subgroup $T$ has only the following possibilities:
\begin{enumerate}
    \item $\mathbb{Z}/N\mathbb{Z}$, where $1\leq N\leq10$ or $N=12$,
    \item $\mathbb{Z}/2\mathbb{Z}\oplus\mathbb{Z}/2N\mathbb{Z}$, where $1\leq N\leq4$.
\end{enumerate}
\end{theorem}
In Section 2, we worked in $\mathbb{C}$. To apply Mordell's and Mazur's theorems, we need to work in $\mathbb{Q}$, that is, we need to work in $E(\mathbb{Q})\subset{E(\mathbb{C})}$. Therefore, we need to show that the elliptic curve $E=\mathcal{J'}^{-1}\circ\mathcal{J}(Q)$ has rational coefficients and $\Omega_3=\mathcal{J'}^{-1}(\omega_3)\in{E(\mathbb{C})}$ is in fact in $E(\mathbb{Q})$.
\begin{lemma}\label{rational-coefficients}
Let $\Lambda$ be the lattice determined by $Q$ as in Theorem \ref{abel-jacobi}. Then,
\begin{equation}
    g_2\left(\Lambda\right)=D(\Delta_1)=D(\Delta_2),
\end{equation}
and
\begin{equation}
    g_3(\Lambda)=-E(\Delta_1)=-E(\Delta_2),
\end{equation}
where $E$ and $D$ are Eisenstein invariants in Definition \ref{eisen}, and $\Delta_1$ and $\Delta_2$ are partial discriminants of $Q$ in Definition \ref{partial-discriminant}.
\end{lemma}
\begin{proof}
See Corrollary 2.4.7 of \cite{duistermaat}.
\end{proof}
\begin{remark}
The lemma says that although both the uniformization of $Q$ by $\Lambda$  and the Abel-Jacobi map of $\mathbb{C}/\Lambda$ are analytic, their composition is completely given by polynomial functions.
\end{remark}
The following lemma shows how the QRT map $\delta$ transforms to a polynomial map under $\mathcal{J}$ and $\mathcal{J'}^{-1}$
\begin{lemma}\label{rational-point}
The addition $\delta^{**}$ on $E(\mathbb{C})$ induced by the QRT map $\delta$ sends the point at infinity $O$ to $(X,Y)$,
where
\begin{equation}
    X = (p_{0,0}^2-4p_{0,-1}p_{0.1}-4p_{-1,0}p_{1,0}+8p_{-1,1}p_{1,-1}+8p_{-1,-1}p_{1,1})/12
\end{equation}
and
\begin{equation}
    Y = -\det{\mathbb{P}},
\end{equation}
where
\begin{equation*}
\mathbb{P}=
\begin{pmatrix}
p_{1,1} & p_{1,0} & p_{1,-1} \\
p_{0,1} & p_{0,0}-1  & p_{0,-1} \\
p_{-1,1} & p_{-1,0} & p_{-1,-1}
\end{pmatrix}
\end{equation*}
\end{lemma}
\begin{proof}
See Proposition 2.5.6 of \cite{duistermaat}.
\end{proof}
\section{Main result}
Gather around all information, we state the main result of this paper.
\begin{theorem}
A finite Galois group $\mathcal{H}$ of the weighted walk with rational coefficients can have order at most 24.
\end{theorem}
\begin{proof}
Since the kernel $Q(x,y)$ has rational coefficients, Lemma \ref{rational-coefficients} says that the associated elliptic curve $E$ in the Weierstrass normal form
\begin{equation*}
    y^2=4x^3-g_2x-g_3
\end{equation*}
also has rational coefficients, i.e. $g_2,g_3\in\mathbb{Q}$.  Lemma \ref{rational-point} says that $\Omega_3=\mathcal{J}'^{-1}(\omega_3)\in{E(\mathbb{Q})}$. Then the group $\langle\Omega_3\rangle$ generated by $\Omega_3$ is a subgroup of $E(\mathbb{Q})$. By Proposition \ref{first-isomorphism} and Theorem \ref{inverse},
\begin{equation*}
    \mathcal{H}_0\cong\langle\Omega_3\rangle.
\end{equation*}
Hence $\mathcal{H}_0\leq{E(\mathbb{Q})}$. Mordell's theorem says that $E(\mathbb{Q})$ is finitely generated, hence by the fundamental theorem of finitely generated Abelian group, we have
\begin{equation}
    \mathcal{H}_0\leq{\mathbb{Z}^{r}\oplus{T}},
\end{equation}
where $r\in\mathbb{Z}_+$ and $T$ is the torsion subgroup. Since $\mathcal{H}$ is assumed to be finite, we have
\begin{equation*}
    \mathcal{H}_0\leq{T}.
\end{equation*}
Mazur's theorem says that $|T|\leq{12}$, so $|\mathcal{H}_0|\leq{12}$. By Equation (\ref{normal}), $|\mathcal{H}|\leq{24}$.
\end{proof}
We rederive two known criteria for the weighted walk having order 4 and 6 using geometric argument.
\begin{theorem}[Criterion for $\mathcal{H}$ of order 4]
$\mathcal{H}$ has order 4 if and only if $\det{\mathbb{P}}=0$.
\end{theorem}
\begin{proof}
$\mathcal{H}$ has order 4 if and only if $\Omega_3$ is a torsion point of order 2 in $E(\mathbb{C})$. The result is obtained by the fact that a point in a Weierstrass normal curve has order 2 if and only if its $Y$ coordinate is 0.
\end{proof}
\begin{theorem}[Criterion for $\mathcal{H}$ of order 6]
$\mathcal{H}$ has order 6 if and only if
\begin{equation}
    \begin{vmatrix}
-12X & 0 & D\\
0 & 1 & Y\\
D & Y & DX+3E
\end{vmatrix}=0
\end{equation}
where $\Omega_3=(X,Y)$ is given by Lemma \ref{rational-point} and $D\coloneqq{D(\Delta_1)}=D(\Delta_2)$, $E\coloneqq{E(\Delta_1)}=E(\Delta_2)$ are Eisenstein invariants given by Lemma \ref{rational-coefficients}.
\end{theorem}
\begin{proof}
$\mathcal{H}$ has order 6 $\Leftrightarrow$ $\Omega_3=((X,Y)$ is a torsion point of order 3 in $E(\mathbb{C})$ $\Leftrightarrow$ $\Omega_3$ is a flex point $\Leftrightarrow$ $\det\left(\operatorname{Hess}(f)\right)$ vanishes on $(X,Y,1)$, where $f(x,y,z)$ is the homogeneous polynomial
\begin{equation*}
    f=y^2z-4x^3+g_2xz^2+g_3z^3
\end{equation*}.
The result is obtained by direct calculation.
\end{proof}
\begin{remark}
The following result for $\mathcal{H}$ to have order 6 is give by Proposition 4.1.8 in \cite{yellowbook}:

$\mathcal{H}$ has order 6 if and only if
\begin{equation*}
    \begin{vmatrix}
\Delta_{11} & \Delta_{21} & \Delta_{12} &\Delta_{22}\\
\Delta_{12} & \Delta_{22} & \Delta_{13}&\Delta_{23}\\
\Delta_{21} & \Delta_{31} & \Delta_{22}&\Delta_{32}\\
\Delta_{22} & \Delta_{32}& \Delta_{23}&\Delta_{33}
\end{vmatrix}=0,
\end{equation*}
where $\Delta_{ij}$'s are cofactors of the matrix $\mathbb{P}$.
\end{remark}
\section{Criterion for orders 4m and 4m+2}
In this section, we give criteria for $\mathcal{H}$ having orders $4m$ or $4m+2$ using division polynomials. The criteria given in \cite{yellowbook} are abstract, requiring linear dependence of certain functions in some function field. Our result is completely given by polynomials.
\begin{definition}[Division polynomials]
Let $y^2=x^3+ax+b$ be an elliptic curve. The division polynomials are
\begin{align*}
    &\psi_1=1,\\
    &\psi_2=2y,\\
    &\psi_3=3x^4+6ax^2+12bx-a^2\\
    &\psi_4=4y(x^6+5ax^4+20bx^3-5a^2x^2-4abx-8b^2-a^3),\\
    &\psi_{2m+1}=\psi_{m+2}\psi_m^3-\psi_{m-1}\psi^3_{m+1},\ m\geq2,\\
    &\psi_{2m}=\frac{1}{2y}\psi_m(\psi_{m+2}\psi_{m-1}^2-\psi_{m-2}\psi_{m+1}^2),\ m\geq3.
\end{align*}
\end{definition}
\begin{remark}
These polynomials get the name because $m|n\Rightarrow{\psi_m|\psi_n}$ in $\mathbb{Z}[x,y,a,b]$ and $(x,y)$ is a torsion point of order dividing $n$ if and only if $(x,y)$ is a zero of $\psi_n$.
\end{remark}
This is the definition commonly found in textbooks, for example page 105 of \cite{silverman}.
To use it in our case $y^2=4x^3-g_2x-g_3$, we need change variables. We believe these formula have been known for a long time but hard to find in literature.
\begin{theorem}[Division polynomials in $g_2$ and $g_3$]
\begin{align*}
    &\Psi_1=1,\\
    &\Psi_2=y,\\
    &\Psi_3=48x^4-24g_2x^2-48g_3x-g_2^2\\
    &\Psi_4=y(64x^6-80g_2x^4-320g_3x^3-20g_2^2x^2-16g_2g_3x+g_2^3-32g_3^2),\\
    &\Psi_{2m+1}=\Psi_{m+2}\Psi_m^3-\Psi_{m-1}\Psi^3_{m+1},\ m\geq2,\\
    &\Psi_{2m}=\frac{1}{y}\Psi_m(\Psi_{m+2}\Psi_{m-1}^2-\Psi_{m-2}\Psi_{m+1}^2),\ m\geq3.
\end{align*}
\end{theorem}
\begin{proof}
Direct change of variables.
\end{proof}
Following corollaries give criteria for $\mathcal{H}$ to have various orders. Their proofs are obvious.
\begin{corollary}[Criterion for $\mathcal{H}$ of order 8]
$\mathcal{H}$ has order 8 if and only if $Y\neq0$ and
\begin{equation*}
    64X^6-80DX^4-320EX^3-20D^2X^2-16DEX+D^3-32E^2=0.
\end{equation*}
where $\Omega_3=(X,Y)$ and $D,E$ are Eisenstein invariants.
\end{corollary}
\begin{remark}
The following result for $\mathcal{H}$ to have order 8 is give by Proposition 4.1.11 in \cite{yellowbook}:

$\mathcal{H}$ has order 8 if and only if
\begin{equation*}
\begin{vmatrix}
A&B &C \\
D&E&F\\
G&H&I
\end{vmatrix}=0,
\end{equation*}
where $A=2\Delta_{22}\Delta_{32}-(\Delta_{21}\Delta_{33}+\Delta_{31}\Delta_{23})$, $B=2(\Delta_{22}^2-\Delta_{12}\Delta_{31}+\Delta_{21}\Delta_{23})+\Delta_{11}\Delta_{33}+\Delta_{31}\Delta_{13}$, $C=2\Delta_{12}\Delta{22}-(\Delta_{11}\Delta_{23}+\Delta_{21})$, $D=\Delta_{32}^2-\Delta_{31}\Delta_{33}$, $E=-2\Delta_{32}\Delta{22}+\Delta_{31}\Delta_{23}+\Delta_{21}\Delta_{33}$, $F=\Delta_{22}^2-\Delta_{21}\Delta_{23}$, $G=\Delta_{22}^2-\Delta_{21}\Delta_{23}$, $H=-2\Delta_{22}\Delta{12}+\Delta_{11}\Delta_{23}+\Delta_{13}\Delta_{21}$, and $I=\Delta_{12}^2-\Delta_{11}\Delta_{13}$, and $\Delta_{ij}$'s are cofactors of the matrix $\mathbb{P}$.

Our result is much simpler.
\end{remark}
\begin{corollary}[Criterion for $\mathcal{H}$ of order $4m$]
$\mathcal{H}$ has order $4m$ if and only if $\Psi_n(X,Y)\neq0$ for all $n|2m, n\neq{2m}$ and $\Psi_{2m}(X,Y)=0$.
\end{corollary}
\begin{corollary}[Criterion for $\mathcal{H}$ of order $4m+2$]
$\mathcal{H}$ has order $4m+2$ if and only if $\Psi_n(X,Y)\neq0$ for all $n|(2m+1), n\neq{2m+1}$ and $\Psi_{2m+1}(X,Y)=0$.
\end{corollary}
\section{Discussion}
We found that the finite $\mathcal{H}$ can have order at most 24 for  weighted walks with rational weights. Geometric proofs of the criterion for $\mathcal{H}$ to have order 4 and 6 are given. In particular for the case of order 6, the result is simpler than Proposition 4.1.8 of \cite{yellowbook}. Using division polynomial, a recursive criterion for $\mathcal{H}$ to have order $4m$ or $4m+2$ is also obtained and explicit criterion for $\mathcal{H}$ to have order 8 is given almost with no computation, much simpler than Proposition 4.1.11 of \cite{yellowbook}.

Since 24 is only an upper bound, further work on finding possible realizations of different group orders is needed.

\bibliographystyle{plain}

\end{document}